\documentclass[12pt,reqno]{amsart}
\usepackage{latexsym,amsmath,amsfonts,amssymb,amsthm,extarrows}
\def\pmod #1{\ ({\rm{mod}}\ #1)}
\def\Z{\mathbb Z}

\def\1{{\mathbf 1}}

\def\pmod #1{\ ({\rm{mod}}\ #1)}

\def\floor #1{\left\lfloor{#1}\right\rfloor}

\theoremstyle{plain}
\newtheorem{Thm}{Theorem}
\newtheorem{Lem}{Lemma}

\theoremstyle{definition}
\newtheorem*{Ack}{Acknowledgment}
\theoremstyle{remark}

\begin{document}

\title{Three supercongruences for Ap\'ery numbers or
Franel numbers}
\author{Yong Zhang}
\email{yongzhang1982@163.com}
\address{Department of Mathematics and Physics, Nanjing Institute of Technology,
Nanjing 211167, People's Republic of China}
\thanks{
The work is supported by National Natural Science Foundation of China (Grant No. 11971222 and 12071208) and  Natural Science Foundation of  Nanjing Institute of Technology (No. CKJB201807).}
\begin{abstract}
The Ap\'ery numbers $A_n$ and the Franel numbers $f_n$ are defined by
$$A_n=\sum_{k=0}^{n}{\binom{n+k}{2k}}^2{\binom{2k}{k}}^2\ \ \ \ \ {\rm and }\ \ \ \ \ \ f_n=\sum_{k=0}^{n}{\binom{n}{k}}^3(n=0, 1, \cdots,).$$
In this paper, we prove three supercongruences for  Ap\'ery numbers or
Franel numbers conjectured by Z.-W. Sun.
Let $p\geq 5$ be a prime and let $n\in \Z^{+}$. We show that
\begin{align} \notag
\frac{1}{n}\bigg(\sum_{k=0}^{pn-1}(2k+1)A_k-p\sum_{k=0}^{n-1}(2k+1)A_k\bigg)\equiv0\pmod{p^{4+3\nu_p(n)}}
\end{align}
and
\begin{align}\notag
\frac{1}{n^3}\bigg(\sum_{k=0}^{pn-1}(2k+1)^3A_k-p^3\sum_{k=0}^{n-1}(2k+1)^3A_k\bigg)\equiv0\pmod{p^{6+3\nu_p(n)}},
\end{align}
where $\nu_p(n)$ denotes the $p$-adic order of $n$.
Also, for any prime $p$ we have 
\begin{align}  \notag
\frac{1}{n^3}\bigg(\sum_{k=0}^{pn-1}(3k+2)(-1)^kf_k-p^2\sum_{k=0}^{n-1}(3k+2)(-1)^kf_k\bigg)\equiv0\pmod{p^{3}}.
\end{align}
\end{abstract}
\maketitle

\section{Introduction}
\setcounter{equation}{0}
\setcounter{Thm}{0}
\setcounter{Lem}{0}
\setcounter{Cor}{0}
\setcounter{Conj}{0}

The Ap\'ery numbers are defined as
$$A_n=\sum_{k=0}^{n}{\binom{n}{k}}^2{\binom{n+k}{k}}^2=\sum_{k=0}^{n}{\binom{n+k}{2k}}^2{\binom{2k}{k}}^2\ \ (n=0, 1, \cdots),$$
which play a central role in Ap\'ery's proof of the irrationality of $\zeta(3)=\sum_{n=1}^{\infty}\frac{1}{n^3}$ (see Ap\'ery \cite{Ap}).

In 2002, Z.-W. Sun \cite{ZWSF2} introduced Ap\'ery polynomials $$A_n(x)=\sum_{k=0}^{n}{\binom{n}{k}}^2{\binom{n+k}{k}}^2x^k=\sum_{k=0}^{n}{\binom{n+k}{2k}}^2{\binom{2k}{k}}^2x^k\ \ (n=0, 1, \cdots).$$
Clearly, $A_n(1)=A_n$. Z.-W. Sun \cite[Theorem 1.1 (ii)]{ZWSF2} showed that for any positive integer $n$ and
integer $x$,
$$\sum_{k=0}^{n-1}(2k+1)A_k(x)\equiv 0\pmod{n}$$
and for any prime $p\geq 5$,
\begin{align}\label{fangc}\sum_{k=0}^{p-1}(2k+1)A_k\equiv p+\frac{7}{6}p^4B_{p-3}\pmod{p^5},\end{align}
where 
the $n$th Bernoulli number $B_n$ is defined by $$B_0=1,\ \ \ \ \ \ \sum_{k=0}^{n-1}\binom{n}{k}B_k=0\ (n\geq 2).$$

Motivated by Z.-W. Sun's work, Guo and Zeng \cite[Theorem 1.3]{VJWGJZF2} made use of some combinatorial identities and $q$-congruences to prove that 
$$\sum_{k=0}^{n-1}(2k+1)^3A_k\equiv0 \pmod{n^3}$$
and
\begin{align} \label{fangc1}\sum_{k=0}^{p-1}(2k+1)^3A_k\equiv p^3 \pmod{p^6}.\end{align}

 In this paper we first prove the following stronger results related
to Ap\'ery numbers conjectured by Z.-W. Sun \cite[(4.2) and (4.3) of Conjecture 56]{ZWSF1}.

\begin{Thm}\label{th1}
Let $p\geq 5$ be a prime and let $n\in \Z^{+}$. Then
\begin{align} \label{fangc2}
\frac{1}{n}\bigg(\sum_{k=0}^{pn-1}(2k+1)A_k-p\sum_{k=0}^{n-1}(2k+1)A_k\bigg)\equiv0\pmod{p^{4+3\nu_p(n)}}
\end{align}
and
\begin{align} \label{fangc6}
\frac{1}{n^3}\bigg(\sum_{k=0}^{pn-1}(2k+1)^3A_k-p^3\sum_{k=0}^{n-1}(2k+1)^3A_k\bigg)\equiv0\pmod{p^{6+3\nu_p(n)}}.
\end{align}
\end{Thm}

Note that the supercongruence (\ref{fangc2}) with $n=1$ yields the slightly weaker version of (\ref{fangc}). Letting $n=1$ in (\ref{fangc6}) gives (\ref{fangc1}). 
Our proofs of (\ref{fangc2}) and (\ref{fangc6}) are based on two identities due to Z.-W. Sun, Guo and Zeng, respectively.

Given a prime $p$ and a positive integer $k$, for a sequence $\{a_n\}_{n\geq 0}$ of integers, the congruence is called Dwork's type congruence, or Atkin and Swinnerton-Dyer type congruence if 
$$
a_{np^{}}\equiv \gamma_p\cdot a_{n}\pmod{p^{k\nu_p(n)}},\qquad\forall n\geq 1.
$$
(\ref{fangc2}) and (\ref{fangc6}) are congruences of this type. One can refer to \cite{ASD, FB, BD} for Dwork's type congruences and \cite{VJWG} for $q$-analogues of Dwork-type supercongruences.

For all nonnegative integers $n$, the Franel numbers are given by
$$f_n=\sum_{k=0}^{n}{\binom{n}{k}}^3,$$ which were first introduced by Franel \cite{JFranel}. In 2013, Guo \cite[Theorems 1.1 and 1.2]{VJWGJZF1} proved that for any $n\in\Z^{+}$ and prime $p\geq 5$,
$$\sum_{k=0}^{n-1}(3k+2)(-1)^kf_k\equiv0 \pmod{n^2}$$
and
$$\sum_{k=0}^{p-1}(3k+2)(-1)^kf_k\equiv2p^2(2^p-1)^2 \pmod{p^5},$$
which were originally conjectured by Z.-W. Sun \cite {ZWSF4}. The second aim of this paper is to prove the following supercongruence, which was conjectured by Z.-W. Sun \cite[(4.5)]{ZWSF1}.
\begin{Thm}\label{th2}Let $p$ be a prime and let $n\in \Z^{+}$. Then
\begin{align} \label{fangc3}
\frac{1}{n^3}\bigg(\sum_{k=0}^{pn-1}(3k+2)(-1)^kf_k-p^2\sum_{k=0}^{n-1}(3k+2)(-1)^kf_k\bigg)\equiv0\pmod{p^{3}}.
\end{align}
\end{Thm}

Our proof of (\ref{fangc3}) makes use of a identity obtained by Guo. The remainder of the paper is organized as follows. In the next section, we give some auxiliary lemmas. The proofs of Theorems 1.1 and 1.2 will be given in Section 3.

\section{Some Lemmas}
\setcounter{equation}{0}
\setcounter{Thm}{0}

\setcounter{Lem}{0}
\setcounter{Cor}{0}
\setcounter{Conj}{0}
In the following section, for an assertion $A$ we adopt the notation:
$$[A]=\begin{cases}1,&\text{if A holds,} \\
0,&\text{otherwise}.\end{cases} $$
We see that $[m=n]$ coincides with the Kronecker symbol $\delta_{m,n}$. In order to prove Theorems 1.1 and 1.2, we first establish the following auxiliary lemmas.
\begin{Lem}
\label{l1}
Let $n, k, r$ be positive integers and $p$ be a prime. Then
\begin{equation}\label{fangc15}
\binom{p^{r}n-1}{k}\equiv \binom{p^{r-1}n-1}{\floor {\frac{k}{p}}}(-1)^{k-\floor {\frac{k}{p}}}\bigg(1-np^{r}\sum_{j=1, p\nmid j}^{k}\frac{1}{j}\bigg)\pmod{p^{2{r}}}.
\end{equation}
\end{Lem}
The congruence (\ref{fangc15}) is the result of Beukers \cite[Lemma 2 (i)]{FB}.
\begin{Lem}\label{l2}\cite[Lemma 2.1 and proof of Theorem 1.3]{ROBSAS}\label{l2} Let $p$ be a prime. Then, for any integers $a, b$ and positive integers $r, s,$ we have
\begin{align}\label{tfangc14}
\binom{p^ra}{p^sb}/\binom{p^{r-1}a}{p^{s-1}b}\equiv(-1)^{(p-1)p^{s-1}b}\pmod {p^{r+s+min\{r, s\}-\delta_{p,3}-2\delta_{p,2}}}.
\end{align}
\end{Lem}

Jacobsthal's binomial congruence (\ref{tfangc14}) with nonnegative $a, b$ was proved by Gessel \cite{IMG} and Granville \cite{AG} for $p\geq 5$ respectively. Straub \cite{AS} showed the extension to negative integers.

The following curious result is also due to  R. Osburn, B. Sahu and A. Straub:
\begin{Lem}[ {\cite[Lemma 2.2]{ROBSAS}}]
\label{l3}
Let $p$ be a prime and n an integer with $p-1\nmid n$. Then, for all integers $r\geq 0$,
\begin{align}\label{fangc24}\sum_{k=1,p\nmid k}^{p^r-1}k^n\equiv 0\pmod{p^r}.
\end{align}
If, additionally, $n$ is even, then, for primes $p\geq 5$,
\begin{align}\label{fangc25}\sum_{k=1,p\nmid k}^{\frac{p^r-1}{2}}\frac{1}{k^n}\equiv 0\pmod{p^r}.
\end{align}
\end{Lem}
\begin{Lem}\label{l4}Let $p\geq 5$ be a prime. For any positive integer $r$ and nonnegative integer $l$, we have
\begin{align}\label{fangc20}\sum_{\lfloor \frac{k}{p^r}\rfloor=l,p\nmid(2k+1)}\frac{1}{2k+1}\equiv 0\pmod{p^{2r
}},
\end{align}
\begin{align}\label{fangc35}
\sum_{\lfloor \frac{k}{p^r}\rfloor=l, p\nmid{(k+1)}}\frac{1}{k+1}\equiv 0\pmod{p^{2r}}
\end{align}
and
\begin{align}\label{fangc44}
\sum_{\lfloor\frac{k}{p^r}\rfloor=l}\sum_{j=1,p\nmid j}^{k}\frac{1}{j^2}\equiv 0\pmod{p^{2r}}.
\end{align}
\end{Lem}
\begin{proof}
Letting $n=-2$ into (\ref{fangc24}) and $n=2$ into (\ref{fangc25}) repectively, we obtain
\begin{align}\label{tfangc9}\sum_{k=1,p\nmid k}^{p^r-1}\frac{1}{k^2}\equiv2\sum_{k=1,p\nmid k}^{\frac{p^r-1}{2}}\frac{1}{k^2}\equiv 0\pmod{p^r}
\end{align}
for any prime $p\geq 5$.
Noting that
\begin{align}\notag\sum_{k=0,p\nmid(2k+1)}^{p^r-1}\frac{1}{2k+1}&=\sum_{j=1,p\nmid j}^{\frac{p^r-1}{2}}\bigg(\frac{1}{p^r+2j}+\frac{1}{p^r-2j}\bigg)\\\notag&=2p^r\sum_{j=1,p\nmid j}^{\frac{p^r-1}{2}}\frac{1}{p^{2r
}-4j^2}\equiv-\frac{p^r}{2}\sum_{j=1,p\nmid j}^{\frac{p^r-1}{2}}\frac{1}{j^2}\pmod{p^{2r
}}
\end{align}
and 
\begin{align}\notag&\sum_{k=0,p\nmid(2k+1)}^{p^r-1}\frac{1}{(2k+1)^2}\\\notag&=\sum_{j=1,p\nmid j}^{\frac{p^r-1}{2}}\bigg(\frac{1}{(p^r+2j)^2}+\frac{1}{(p^r-2j)^2}\bigg)\equiv \frac{1}{2}\sum_{j=1,p\nmid j}^{\frac{p^r-1}{2}}\frac{1}{j^2}\pmod{p^{r
}},
\end{align} and using (\ref{tfangc9}), we arrive at
\begin{align}\notag\sum_{\lfloor \frac{k}{p^r}\rfloor=l,p\nmid(2k+1)}\frac{1}{2k+1}&=\sum_{k=0,p\nmid(2k+1)}^{p^r-1}\frac{1}{2p^rl+2k+1}\\\notag&\equiv\sum_{k=0,p\nmid(2k+1)}^{p^r-1}\frac{1}{2k+1}\bigg(1-\frac{2p^rl}{2k+1}\bigg)\equiv 0\pmod{p^{2r
}}.
\end{align}
This proves (\ref{fangc20}).

Recall that \cite[Lemma 1]{FB}:
\begin{align}\label{tfangc10}
\sum_{\lfloor \frac{k}{p^r}\rfloor=l, p\nmid{k}}\frac{1}{k}\equiv 0\pmod{p^{2r}}.
\end{align}
Hence
\begin{align}\notag
\sum_{\lfloor \frac{k}{p^r}\rfloor=l, p\nmid{(k+1)}}\frac{1}{k+1}&=\sum_{k=1, p\nmid{k}}^{p^r-1}\frac{1}{p^rl+k}=\sum_{\lfloor \frac{k}{p^r}\rfloor=l, p\nmid{k}}\frac{1}{k}\equiv 0\pmod{p^{2r}}
\end{align}
as desired. 

Observe that 
\begin{align}\notag
\sum_{k=1}^{p^r-1}\sum_{j=1,p\nmid j}^{k}\frac{1}{(p^rl+j)^2}&\equiv\sum_{k=1}^{p^r-1}\sum_{j=1,p\nmid j}^{k}\frac{1}{j^2+2p^rl}\equiv\sum_{k=1}^{p^r-1}\sum_{j=1,p\nmid j}^{k}\bigg(\frac{1}{j^2}-\frac{2p^rl}{j^4}\bigg)\\\notag&=\sum_{j=1,p\nmid j}^{p^r-1}\bigg(\frac{1}{j^2}-\frac{2p^rl}{j^4}\bigg)\sum_{k=j}^{p^r-1}1=\sum_{j=1,p\nmid j}^{p^r-1}\bigg(\frac{1}{j^2}-\frac{2p^rl}{j^4}\bigg)(p^r-j)\\\label{tfangc8}&\equiv p^r\sum_{j=1,p\nmid j}^{p^r-1}\frac{1}{j^2}-\sum_{j=1,p\nmid j}^{p^r-1}\frac{1}{j}+\sum_{j=1,p\nmid j}^{p^r-1}\frac{2p^rl}{j^3}\pmod{p^{2r}}
\end{align}
and 
\begin{align}\label{tfangc7}
\sum_{k=0}^{p^r-1}\sum_{j=1,p\nmid j}^{p^rl-1}\frac{1}{j^2}=p^r\sum_{k=0}^{l-1}\sum_{j=1,p\nmid j}^{p^r-1}\frac{1}{(p^rk+j)^2}\equiv p^rl\sum_{j=1,p\nmid j}^{p^r-1}\frac{1}{j^2}\pmod{p^{2r}}.
\end{align}
Combining (\ref{fangc24}) in the case $n=-3$ and (\ref{tfangc9})-(\ref{tfangc7}), we obtain
\begin{align}\notag
\sum_{\lfloor\frac{k}{p^r}\rfloor=l}\sum_{j=1,p\nmid j}^{k}\frac{1}{j^2}&=\sum_{k=0}^{p^r-1}\sum_{j=1,p\nmid j}^{p^rl+k}\frac{1}{j^2}\\\notag&=\sum_{k=0}^{p^r-1}\sum_{j=1,p\nmid j}^{p^rl-1}\frac{1}{j^2}+\sum_{k=1}^{p^r-1}\sum_{j=1,p\nmid j}^{k}\frac{1}{(p^rl+j)^2}\equiv 0\pmod{p^{2r}}.
\end{align}
So (\ref{fangc44}) is valid.

Now the proof of Lemma \ref{l4} is complete.
\end{proof}

\section{Proofs of Theorems \ref{th1} and \ref{th2}}
\setcounter{equation}{0}
\setcounter{Thm}{0}
\setcounter{Lem}{0}
\setcounter{Cor}{0}
\setcounter{Conj}{0}

\begin{proof}[Proof of Theorem \ref{th1}]
In order to prove (\ref{fangc2}), we need the following combinatorial identity due to Z.-W. Sun
\cite[(1.5)]{ZWSF2}. For any positive integer $n$ we have
\begin{align}\notag
\frac{1}{n}\sum_{k=0}^{n-1}(2k+1)A_k(x)&=\sum_{k=0}^{n-1}\binom{n-1}k\binom{n+k}k\binom{n+k}{2k+1}\binom{2k}kx^k\\\label{fangc16}&=\sum_{k=0}^{n-1}\frac{n}{2k+1}{\binom{n-1}k}^2{\binom{n+k}k}^2x^k.
\end{align}
Let $x=1$ and $n=p^{r-1+j}m$ with $j\in\{0, 1\}$ in (\ref{fangc16}), where $r, m\in \Z^{+}$ and $p\nmid m$. Thus,
\begin{align}\notag
&\frac{1}{p^{r-1}m}\bigg(\sum_{k=0}^{p^rm-1}(2k+1)A_k-p\sum_{k=0}^{p^{r-1}m-1}(2k+1)A_k\bigg)\\\notag&=mp^{r+1}\sum_{k=0}^{p^rm-1}{\binom{p^rm-1}k}^2{\binom{p^rm+k}k}^2\frac{1}{2k+1}\\\label{tfangc1}&\ \ \ -mp^{r}\sum_{k=0}^{p^{r-1}m-1}{\binom{p^{r-1}m-1}k}^2{\binom{p^{r-1}m+k}k}^2\frac{1}{2k+1}.
\end{align}
Since for $0\leq s \leq r-1$ and $1\leq k \leq p^{r-s}m-1$, we get\begin{align}\notag&
\binom{p^{r-s}m-1}k\binom{p^{r-s}m+k}k\\\notag&=\prod_{j=1,p\mid j}^{k}\frac{p^{r-s}m-j}{j}\frac{p^{r-s}m+j}{j}\prod_{j=1,p\nmid j}^{k}\frac{p^{r-s}m-j}{j}\frac{p^{r-s}m+j}{j}\\\label{fangc23}&\equiv\binom{p^{r-s-1}m-1}{\lfloor \frac{k}{p}\rfloor}\binom{p^{r-s-1}m+\lfloor \frac{k}{p}\rfloor}{\lfloor \frac{k}{p}\rfloor}(-1)^{k-\lfloor \frac{k}{p}\rfloor}\bigg(1-\sum_{j=1,p\nmid j}^{k}\frac{p^{2r-2s}m^2}{j^2}\bigg)\pmod{p^{4r-4s
}}.
\end{align}
Applying (\ref{fangc23}) with $s=0$, we arrive at 
\begin{align}\notag&mp^{r+1}\sum_{k=0,p\nmid(2k+1)}^{p^rm-1}{\binom{p^rm-1}k}^2{\binom{p^rm+k}k}^2\frac{1}{2k+1}\\\notag&\equiv mp^{r+1}\sum_{k=0,p\nmid(2k+1)}^{p^rm-1}{\binom{p^{r-1}m-1}{\lfloor \frac{k}{p}\rfloor}}^2{\binom{p^{r-1}m+\lfloor \frac{k}{p}\rfloor}{\lfloor \frac{k}{p}\rfloor}}^2\frac{1}{2k+1}\\\notag&=mp^{r+1}\sum_{l=0}^{p^{r-1}m-1}{\binom{p^{r-1}m-1}{l}}^2{\binom{p^{r-1}m+l}{l}}^2\sum_{\lfloor \frac{k}{p}\rfloor=l,p\nmid(2k+1)}\frac{1}{2k+1}\pmod{p^{3r+1
}}.
\end{align}
By (\ref{fangc20}) (with $r=1$) and (\ref{fangc23}) (with $s=1$),
\begin{align}\notag&mp^{r+1}\sum_{k=0,p\nmid(2k+1)}^{p^rm-1}{\binom{p^rm-1}k}^2{\binom{p^rm+k}k}^2\frac{1}{2k+1}\\\notag&\equiv mp^{r+1}\sum_{l=0}^{p^{r-2}m-1}{\binom{p^{r-2}m-1}{l}}^2{\binom{p^{r-2}m+l}{l}}^2\sum_{\lfloor \frac{k}{p^2}\rfloor=l,p\nmid(2k+1)}\frac{1}{2k+1}\pmod{p^{3r+1
}}.
\end{align}
Combining (\ref{fangc20}) and (\ref{fangc23}), by induction, for $3\leq s \leq r-1$ we have 
\begin{align}\notag&mp^{r+1}\sum_{k=0,p\nmid(2k+1)}^{p^rm-1}{\binom{p^rm-1}k}^2{\binom{p^rm+k}k}^2\frac{1}{2k+1}\\\notag&\equiv mp^{r+1}\sum_{l=0}^{p^{r-s}m-1}{\binom{p^{r-s}m-1}{l}}^2{\binom{p^{r-s}m+l}{l}}^2\sum_{\lfloor \frac{k}{p^s}\rfloor=l,p\nmid(2k+1)}\frac{1}{2k+1}\\\label{tfangc2}&\equiv mp^{r+1}\sum_{l=0}^{m-1}{\binom{m-1}{l}}^2{\binom{m+l}{l}}^2\sum_{\lfloor \frac{k}{p^r}\rfloor=l,p\nmid(2k+1)}\frac{1}{2k+1}\equiv0\pmod{p^{3r+1
}}.
\end{align}
On the other hand,
\begin{align}\notag
&mp^{r+1}\sum_{k=0,p\mid(2k+1)}^{p^rm-1}{\binom{p^rm-1}k}^2{\binom{p^rm+k}k}^2\frac{1}{2k+1}\\\label{tfangc3}&=mp^{r}\sum_{l=0}^{p^{r-1}m-1}{\binom{p^rm-1}{lp+\frac{p-1}{2}}}^2{\binom{p^rm+lp+\frac{p-1}{2}}{lp+\frac{p-1}{2}}}^2\frac{1}{2l+1}.
\end{align}
For $l\in \{0, \cdots, p^{r-1}m-1\}$, we split the above binomial sum into two cases $p\mid 2l+1$ and $p\nmid 2l+1$. Letting $s=0$ and $k=lp+\frac{p-1}{2}$ in (\ref{fangc23}) and noting (\ref{tfangc3}), we obtain
\begin{align}\notag
&mp^{r}\sum_{l=0, p\nmid (2l+1)}^{p^{r-1}m-1}{\binom{p^rm-1}{lp+\frac{p-1}{2}}}^2{\binom{p^rm+lp+\frac{p-1}{2}}{lp+\frac{p-1}{2}}}^2\frac{1}{2l+1}\\\notag&\equiv mp^{r}\sum_{l=0, p\nmid (2l+1)}^{p^{r-1}m-1}{\binom{p^{r-1}m-1}{l}}^2{\binom{p^{r-1}m+l}{l}}^2\\\notag&\ \ \ \ \times\bigg(1-\sum_{j=1,p\nmid j}^{{lp+\frac{p-1}{2}}}\frac{2p^{2r}m^2}{j^2}\bigg)\frac{1}{2l+1}\pmod{p^{3r+1
}}.
\end{align}
Now, for any positive integer $s$, by (\ref{tfangc9}) and (\ref{tfangc7}) we have
\begin{align}\label{fangc26}\sum_{j=1,p\nmid j}^{{lp^s+\frac{p^s-1}{2}}}\frac{1}{j^2}&=\sum_{j=1,p\nmid j}^{{lp^s-1}}\frac{1}{j^2}+\sum_{j=1,p\nmid j}^{\frac{p^s-1}{2}}\frac{1}{(lp^s+j)^2}\equiv l\sum_{j=1,p\nmid j}^{{p^s-1}}\frac{1}{j^2}+\sum_{j=1,p\nmid j}^{\frac{p^s-1}{2}}\frac{1}{j^2}\equiv0\pmod{p^{s}}.
\end{align}
It follows that
\begin{align}\notag
&mp^{r}\sum_{l=0, p\nmid (2l+1)}^{p^{r-1}m-1}{\binom{p^rm-1}{lp+\frac{p-1}{2}}}^2{\binom{p^rm+lp+\frac{p-1}{2}}{lp+\frac{p-1}{2}}}^2\frac{1}{2l+1}\\\label{tfangc4}&\equiv mp^{r}\sum_{l=0, p\nmid (2l+1)}^{p^{r-1}m-1}{\binom{p^{r-1}m-1}{l}}^2{\binom{p^{r-1}m+l}{l}}^2\frac{1}{2l+1}\pmod{p^{3r+1
}}.
\end{align}
From (\ref{tfangc1}), (\ref{tfangc2}), (\ref{tfangc3}) and (\ref{tfangc4}), we deduce that
\begin{align}\notag
&\frac{1}{mp^{r-1}}\bigg(\sum_{k=0}^{p^rm-1}(2k+1)A_k-p\sum_{k=0}^{p^{r-1}m-1}(2k+1)A_k\bigg)\\\notag&\equiv\sum_{k=0, p\mid (2k+1)}^{p^{r-1}m-1}\frac{ mp^{r}}{2k+1}\bigg({\binom{p^rm-1}{kp+\frac{p-1}{2}}}^2{\binom{p^rm+kp+\frac{p-1}{2}}{kp+\frac{p-1}{2}}}^2\\\label{tfangc27}&\ \ \ -{\binom{p^{r-1}m-1}{k}}^2{\binom{p^{r-1}m+k}{k}}^2\bigg)\pmod{p^{3r+1
}}.
\end{align}
Letting $k=lp+\frac{p-1}{2}$ in the sums of the right-hand side of (\ref{tfangc27}), we obtain
\begin{align}\notag
&\frac{1}{mp^{r-1}}\bigg(\sum_{k=0}^{p^rm-1}(2k+1)A_k-p\sum_{k=0}^{p^{r-1}m-1}(2k+1)A_k\bigg)\\\notag&=\sum_{l=0}^{p^{r-2}m-1}\frac{ mp^{r-1}}{2l+1}\bigg({\binom{p^rm-1}{lp^2+\frac{p^2-1}{2}}}^2{\binom{p^rm+lp^2+\frac{p^2-1}{2}}{lp^2+\frac{p^2-1}{2}}}^2\\\label{fangc27}&\ \ \ -{\binom{p^{r-1}m-1}{lp+\frac{p-1}{2}}}^2{\binom{p^{r-1}m+lp+\frac{p-1}{2}}{lp+\frac{p-1}{2}}}^2\bigg)\pmod{p^{3r+1
}}.
\end{align}
For any positive integer $s\in \{2, \cdots, r-1\}$, we get
\begin{align}\notag
&\sum_{l=0, p\mid (2l+1)}^{p^{r-s}m-1}\frac{ mp^{r-s+1}}{2l+1}\bigg({\binom{p^rm-1}{lp^s+\frac{p^s-1}{2}}}^2{\binom{p^rm+lp^s+\frac{p^s-1}{2}}{lp^s+\frac{p^s-1}{2}}}^2\\\notag&\ \ \ -{\binom{p^{r-1}m-1}{lp^{s-1}+\frac{p^{s-1}-1}{2}}}^2{\binom{p^{r-1}m+lp^{s-1}+\frac{p^{s-1}-1}{2}}{lp^{s-1}+\frac{p^{s-1}-1}{2}}}^2\bigg)\\\notag&=\sum_{k=0}^{p^{r-s-1}m-1}\frac{ mp^{r-s}}{2k+1}\bigg({\binom{p^rm-1}{kp^{s+1}+\frac{p^{s+1}-1}{2}}}^2{\binom{p^rm+kp^{s+1}+\frac{p^{s+1}-1}{2}}{kp^{s+1}+\frac{p^{s+1}-1}{2}}}^2\\\label{fangc28}&\ \ \ -{\binom{p^{r-1}m-1}{kp^{s}+\frac{p^{s}-1}{2}}}^2{\binom{p^{r-1}m+kp^{s}+\frac{p^{s}-1}{2}}{kp^{s}+\frac{p^{s}-1}{2}}}^2\bigg).
\end{align}
By (\ref{fangc23}) and (\ref{fangc26}), for any positive integer $s\in \{2, \cdots, r\}$ we have
\begin{align}\notag
&\sum_{l=0, p\nmid (2l+1)}^{p^{r-s}m-1}\frac{ mp^{r-s+1}}{2l+1}\bigg({\binom{p^rm-1}{lp^s+\frac{p^s-1}{2}}}^2{\binom{p^rm+lp^s+\frac{p^s-1}{2}}{lp^s+\frac{p^s-1}{2}}}^2\\\notag&\ \ \ -{\binom{p^{r-1}m-1}{lp^{s-1}+\frac{p^{s-1}-1}{2}}}^2{\binom{p^{r-1}m+lp^{s-1}+\frac{p^{s-1}-1}{2}}{lp^{s-1}+\frac{p^{s-1}-1}{2}}}^2\bigg)\\\notag&\equiv-2p^{3r-s+1}m^3\sum_{l=0, p\nmid (2l+1)}^{p^{r-s}m-1}\frac{1}{2l+1}{\binom{p^{r-1}m-1}{lp^{s-1}+\frac{p^{s-1}-1}{2}}}^2{\binom{p^{r-1}m+lp^{s-1}+\frac{p^{s-1}-1}{2}}{lp^{s-1}+\frac{p^{s-1}-1}{2}}}^2\\\label{fangc29}&\ \ \ \times\sum_{j=1,p\nmid j}^{{lp^s+\frac{p^s-1}{2}}}\frac{1}{j^2}\equiv 0\pmod{p^{3r+1
}}.
\end{align}
Combining (\ref{fangc28}), (\ref{fangc29}) with (\ref{fangc27}), by induction, we deduce that
\begin{align}\notag
&\frac{1}{mp^{r-1}}\bigg(\sum_{k=0}^{p^rm-1}(2k+1)A_k-p\sum_{k=0}^{p^{r-1}m-1}(2k+1)A_k\bigg)\\\notag&\equiv\sum_{k=0}^{m-1}\frac{ mp}{2k+1}\bigg({\binom{p^rm-1}{kp^{r}+\frac{p^{r}-1}{2}}}^2{\binom{p^rm+kp^{r}+\frac{p^{r}-1}{2}}{kp^{r}+\frac{p^{r}-1}{2}}}^2\\\label{fangc30}&\ \ \ -{\binom{p^{r-1}m-1}{kp^{r-1}+\frac{p^{r-1}-1}{2}}}^2{\binom{p^{r-1}m+kp^{r-1}+\frac{p^{r-1}-1}{2}}{kp^{r-1}+\frac{p^{r-1}-1}{2}}}^2\bigg)\pmod{p^{3r+1
}}.
\end{align}
If $m-1<\frac{p-1}{2}$, namely, for $k\in\{0, \cdots, m-1\}$ we have $p\nmid (2k+1)$. The congruence (\ref{fangc30}) follows from (\ref{fangc29}) with $s=r$. While $m-1\geq \frac{p-1}{2}$, there are some $k$ with $p\nmid (2k+1)$ or  $p\mid (2k+1)$. The proof of the case $p\nmid (2k+1)$ is very similar to the proof of (\ref{fangc29}), only requiring a few additional discussions with the case $p\mid (2k+1)$. Therefore,
\begin{align}\notag
&\frac{1}{mp^{r-1}}\bigg(\sum_{k=0}^{p^rm-1}(2k+1)A_k-p\sum_{k=0}^{p^{r-1}m-1}(2k+1)A_k\bigg)\\\notag&\equiv\sum_{l=0}^{\lfloor \frac{m-1-(p-1)/2}{p}\rfloor}\frac{ m}{2l+1}\bigg({\binom{p^rm-1}{lp^{r+1}+\frac{p^{r+1}-1}{2}}}^2{\binom{p^rm+lp^{r+1}+\frac{p^{r+1}-1}{2}}{lp^{r+1}+\frac{p^{r+1}-1}{2}}}^2\\\label{fangc31}&\ \ \ -{\binom{p^{r-1}m-1}{lp^{r}+\frac{p^{r}-1}{2}}}^2{\binom{p^{r-1}m+lp^{r}+\frac{p^{r}-1}{2}}{lp^{r}+\frac{p^{r}-1}{2}}}^2\bigg)\pmod{p^{3r+1
}}.
\end{align}
By induction, in the end, there must be such a case for all integers $l\in \{0,\cdots, m_1\}$ with $p\nmid (2l+1)$.
In this case, we can rewrite the sum on the right-hand side of (\ref{fangc31}) as 

\begin{align}\notag&\sum_{l=0}^{m_1}\frac{ m}{(2l+1)p^t}\bigg({\binom{p^rm-1}{lp^{r+t+1}+\frac{p^{r+t+1}-1}{2}}}^2{\binom{p^rm+lp^{r+t+1}+\frac{p^{r+t+1}-1}{2}}{lp^{r+t+1}+\frac{p^{r+t+1}-1}{2}}}^2\\\label{tfangc5}& \ \ -{\binom{p^{r-1}m-1}{lp^{r+t}+\frac{p^{r+t}-1}{2}}}^2{\binom{p^{r-1}m+lp^{r}+\frac{p^{r+t}-1}{2}}{lp^{r+t}+\frac{p^{r+t}-1}{2}}}^2\bigg)\pmod{p^{3r+1
}}.\end{align}
Substituting (\ref{fangc23}) (with $k=lp^{r+t+1}+\frac{p^{r+t+1}-1}{2}$ and $s=0$) and (\ref{fangc26}) (with $s=r+t+1$) into (\ref{tfangc5}), we immediately get (\ref{fangc2}).

Next we will prove (\ref{fangc6}). A nice identity of Guo and Zeng \cite[(4.8)]{VJWGJZF2} ) states that
\begin{align}\notag\sum_{k=0}^{n-1}(2k+1)^3 A_k&=n^2\sum_{k=0}^{n-1}\binom{n+k}{k}{\binom{n-1}{k}}^2\bigg(2n\binom{n+k}{k+1}-\binom{n+k}{k}\bigg)\\\label{fangc32}&=n^2\sum_{k=0}^{n-1}{\binom{n+k}{k}}^2{\binom{n-1}{k}}^2\bigg(\frac{2n^2}{k+1}-1\bigg).
\end{align}
Substituting $n=p^{r-1+j}m$ with $j\in\{0, 1\}$ in (\ref{fangc32}), where $r, m\in \Z^{+}$ and $p\nmid m$, we have
\begin{align}\notag
&\frac{1}{m^3p^{3r}}\bigg(\sum_{k=0}^{p^rm-1}(2k+1)^3A_k-p^3\sum_{k=0}^{p^{r-1}m-1}(2k+1)^3A_k\bigg)\\\notag&=\sum_{k=0}^{p^rm-1}{\binom{p^rm+k}{k}}^2{\binom{p^rm-1}{k}}^2\bigg(\frac{2p^rm}{k+1}-\frac{1}{p^rm}\bigg)\\\label{fangc33}&\ \ \ \ -\sum_{k=0}^{p^{r-1}m-1}{\binom{p^{r-1}m+k}{k}}^2{\binom{p^{r-1}m-1}{k}}^2\bigg(\frac{2p^{r-1}m}{k+1}-\frac{1}{p^{r-1}m}\bigg).
\end{align}
By (\ref{fangc23}) (with $s=0$), we obtain that
\begin{align}\notag
&\sum_{k=0, p\nmid{(k+1)}}^{p^rm-1}{\binom{p^rm+k}{k}}^2{\binom{p^rm-1}{k}}^2\frac{1}{k+1}\\\label{fangc34}&\equiv\sum_{l=0}^{p^{r-1}m-1}{\binom{p^{r-1}m+l}{l}}^2{\binom{p^{r-1}m-1}{l}}^2\sum_{\lfloor \frac{k}{p}\rfloor=l, p\nmid{(k+1)}}\frac{1}{k+1}\pmod{p^{2r}}.
\end{align}
For any integer $1\leq s \leq r-1$, by (\ref{fangc35}) and (\ref{fangc23}), we get
\begin{align}\notag&\sum_{l=0}^{p^{r-s}m-1}{\binom{p^{r-s}m+l}{l}}^2{\binom{p^{r-s}m-1}{l}}^2\sum_{\lfloor \frac{k}{p^s}\rfloor=l, p\nmid{(k+1)}}\frac{1}{k+1}\\\label{tfangc11}&\equiv\sum_{l=0}^{p^{r-s-1}m-1}{\binom{p^{r-s-1}m+l}{l}}^2{\binom{p^{r-s-1}m-1}{l}}^2\sum_{\lfloor \frac{k}{p^{s+1}}\rfloor=l, p\nmid{(k+1)}}\frac{1}{k+1}\pmod{p^{2r}}.
\end{align}
Combining (\ref{fangc34}) with (\ref{tfangc11}), by induction, we see that
\begin{align}\notag
&\sum_{k=0, p\nmid{(k+1)}}^{p^rm-1}{\binom{p^rm+k}{k}}^2{\binom{p^rm-1}{k}}^2\frac{1}{k+1}\\\notag&\equiv\sum_{l=0}^{m-1}{\binom{m+l}{l}}^2{\binom{m-1}{l}}^2\sum_{\lfloor \frac{k}{p^r}\rfloor=l, p\nmid{(k+1)}}\frac{1}{k+1}\equiv 0\pmod{p^{2r}}.
\end{align}
Note that 
\begin{align}\notag
&\sum_{k=0, p\mid{(k+1)}}^{p^rm-1}{\binom{p^rm+k}{k}}^2{\binom{p^rm-1}{k}}^2\frac{2p^rm}{k+1}\\\notag&\ \ \ \ \ \ \ \ \ \ \ \ \ \ \ \ \ \ \ \ \ -\sum_{k=0}^{p^{r-1}m-1}{\binom{p^{r-1}m+k}{k}}^2{\binom{p^{r-1}m-1}{k}}^2\frac{2p^{r-1}m}{k+1}\\\notag&=\sum_{l=0}^{p^{r-1}m-1}\bigg({\binom{p^{r}m+pl+p-1}{pl+p-1}}^2{\binom{p^{r}m-1}{pl+p-1}}^2\\\label{fangc37}&\ \ \ \ \ \ \ \ \ \ \ \ \ \ \ \ \ \ \ \ \  -{\binom{p^{r-1}m+l}{l}}^2{\binom{p^{r-1}m-1}{l}}^2\bigg)\frac{2p^{r-1}m}{l+1}.
\end{align}
Since for any nonnegative integer $l$, by (\ref{tfangc7}) we obtain
\begin{align}\label{fangc38}\sum_{j=1,p\nmid j}^{p^sl+p^s-1}\frac{1}{j^2}\equiv (l+1)\sum_{j=1, p\nmid j}^{p^s-1}\frac{1}{j^2}\equiv 0\pmod{p^{s}}.
\end{align}
For any positive integer $s\leq r$, setting $k=p^sl+p^s-1$ in (\ref{fangc23}) and using (\ref{fangc38}) yields
\begin{align}\notag 
&\sum_{k=0, p\nmid(k+1) }^{p^{r-s}m-1}\bigg({\binom{p^{r}m+p^sk+p^s-1}{p^sk+p^s-1}}^2{\binom{p^{r}m-1}{p^sk+p^s-1}}^2\\\notag&\ \ \ \ \ \ \ \ \ \ \ \ \ \ \ \ \ \  -{\binom{p^{r-1}m+p^{s-1}k+p^{s-1}-1}{p^{s-1}k+p^{s-1}-1}}^2{\binom{p^{r-1}m-1}{p^{s-1}k+p^{s-1}-1}}^2\bigg)\frac{2p^{r-s}m}{k+1}\\\notag&\equiv-4p^{3r-s}m^3\sum_{k=0, p\nmid(k+1) }^{p^{r-s}m-1}{\binom{p^{r-1}m+p^{s-1}k+p^{s-1}-1}{p^{s-1}k+p^{s-1}-1}}^2{\binom{p^{r-1}m-1}{p^{s-1}k+p^{s-1}-1}}^2\\\label{fangc39}&\ \ \ \ \ \ \ \ \ \ \  \ \ \ \ \ \ \times\frac{1}{k+1}\sum_{j=1, p\nmid j}^{p^sk+p^s-1}\frac{1}{j^2}\equiv0\pmod{p^{3r
}}.
\end{align}
Note that for any positive integer $s\leq r-1,$
\begin{align}\notag 
&\sum_{k=0, p\mid(k+1) }^{p^{r-s}m-1}\bigg({\binom{p^{r}m+p^sk+p^s-1}{p^sk+p^s-1}}^2{\binom{p^{r}m-1}{p^sk+p^s-1}}^2\\\notag&\ \ \ \ \ \ \ \ \ \ \ \ \ \ \ \ \ \  -{\binom{p^{r-1}m+p^{s-1}k+p^{s-1}-1}{p^{s-1}k+p^{s-1}-1}}^2{\binom{p^{r-1}m-1}{p^{s-1}k+p^{s-1}-1}}^2\bigg)\frac{2p^{r-s}m}{k+1}\\\notag&=
\sum_{l=0}^{p^{r-s-1}m-1}\bigg({\binom{p^{r}m+p^{s+1}l+p^{s+1}-1}{p^{s+1}l+p^{s+1}-1}}^2{\binom{p^{r}m-1}{p^{s+1}l+p^{s+1}-1}}^2\\\label{tfangc12}&\ \ \ \ \ \ \ \ \ \ \ \ \ \ \ \ \ \  -{\binom{p^{r-1}m+p^{s}l+p^{s}-1}{p^{s}l+p^{s}-1}}^2{\binom{p^{r-1}m-1}{p^{s}l+p^{s}-1}}^2\bigg)\frac{2p^{r-s-1}m}{l+1}.
\end{align}
Combining (\ref{fangc37}), (\ref{fangc39}) and (\ref{tfangc12}), by induction, we get
\begin{align}\notag
&\sum_{k=0, p\mid{(k+1)}}^{p^rm-1}{\binom{p^rm+k}{k}}^2{\binom{p^rm-1}{k}}^2\frac{2p^rm}{k+1}\\\notag&\ \ \ \ \ \ \ \ \ \ \ \ \ \ \ \ \ \ -\sum_{k=0}^{p^{r-1}m-1}{\binom{p^{r-1}m+k}{k}}^2{\binom{p^{r-1}m-1}{k}}^2\frac{2p^{r-1}m}{k+1}\\\notag&\equiv\sum_{k=0, p\mid(k+1) }^{m-1}\bigg({\binom{p^{r}m+p^rk+p^r-1}{p^rk+p^r-1}}^2{\binom{p^{r}m-1}{p^rk+p^r-1}}^2\\\notag&\ \ \ \ \ \ \ \ \ \ \ \ \ \ \ \ \ \ -{\binom{p^{r-1}m+p^{r-1}k+p^{r-1}-1}{p^{r-1}k+p^{r-1}-1}}^2{\binom{p^{r-1}m-1}{p^{r-1}k+p^{r-1}-1}}^2\bigg)\frac{2m}{k+1}\pmod{p^{3r}}.
\end{align}
If $m\leq p-1$, by (\ref{fangc39}) with $s=r$, we have
 \begin{align}\notag
&\sum_{k=0, p\mid{(k+1)}}^{p^rm-1}{\binom{p^rm+k}{k}}^2{\binom{p^rm-1}{k}}^2\frac{2p^rm}{k+1}\\\notag&\ \ \ \ \ \ \ \ \ \ \ \ \ \ \ \ \ \ -\sum_{k=0}^{p^{r-1}m-1}{\binom{p^{r-1}m+k}{k}}^2{\binom{p^{r-1}m-1}{k}}^2\frac{2p^{r-1}m}{k+1}\equiv 0\pmod{p^{3r}}.
\end{align}
While $m\geq p$, without loss of generality, suppose that $m=\sum_{k=0}^{t}m_kp^{k}$ with $ m_0,  m_t\in\{1, \cdots, p-1\}$ and $m_k\in\{0, \cdots, p-1\}$ for $k\in\{1, \cdots, t-1\}$, by induction, we get 

\begin{align}\notag
&\sum_{k=0, p\mid(k+1) }^{m-1}\bigg({\binom{p^{r}m+p^rk+p^r-1}{p^rk+p^r-1}}^2{\binom{p^{r}m-1}{p^rk+p^r-1}}^2\\\notag&\ \ \ -{\binom{p^{r-1}m+p^{r-1}k+p^{r-1}-1}{p^{r-1}k+p^{r-1}-1}}^2{\binom{p^{r-1}m-1}{p^{r-1}k+p^{r-1}-1}}^2\bigg)\frac{2m}{k+1}\\\notag&\equiv\sum_{l=0 }^{\lfloor \frac{m-p}{p}\rfloor}\bigg({\binom{p^{r}m+p^{r+1}l+p^{r+1}-1}{p^{r+1}l+p^{r+1}-1}}^2{\binom{p^{r}m-1}{p^{r+1}l+p^{r+1}-1}}^2\\\notag&\ \ \ -{\binom{p^{r-1}m+p^{r}l+p^{r}-1}{p^{r}l+p^{r}-1}}^2{\binom{p^{r-1}m-1}{p^{r}l+p^{r}-1}}^2\bigg)\frac{2m}{p(l+1)}\\\notag&\equiv\sum_{l=0 }^{m_t-1}\bigg({\binom{p^{r}m+p^{r+t}l+p^{r+t}-1}{p^{r+t}l+p^{r+t}-1}}^2{\binom{p^{r}m-1}{p^{r+t}l+p^{r+t}-1}}^2-{\binom{p^{r-1}m-1}{p^{r+t-1}l+p^{r+t-1}-1}}^2\\\notag&\ \ \times{\binom{p^{r-1}m+p^{r+t-1}l+p^{r+t-1}-1}{p^{r+t-1}l+p^{r+t-1}-1}}^2\bigg)\frac{2m}{p^{{t}}(l+1)}\equiv 0\pmod{p^{3r}},
\end{align}
where the last conguence comes from
(\ref{fangc39}) with $s=r+t$, since $p\nmid (l+1)$
for $0\leq l\leq m_t-1 \leq p-2$.
Combining the above with (\ref{fangc33}), it suffices to show that
\begin{align}\notag&\sum_{k=0}^{p^{r}m-1}{\binom{p^{r}m+k}{k}}^2{\binom{p^{r}m-1}{k}}^2\\\label{fangc41}&\ \ -p\sum_{k=0}^{p^{r-1}m-1}{\binom{p^{r-1}m+k}{k}}^2{\binom{p^{r-1}m-1}{k}}^2\equiv 0\pmod{p^{4r}}.
\end{align}
Substituting (\ref{fangc23}) with $s=0$ into (\ref{fangc41}) gives
\begin{align}\notag&\sum_{k=0}^{p^{r}m-1}{\binom{p^{r}m+k}{k}}^2{\binom{p^{r}m-1}{k}}^2-p\sum_{k=0}^{p^{r-1}m-1}{\binom{p^{r-1}m+k}{k}}^2{\binom{p^{r-1}m-1}{k}}^2\\\notag&\equiv\sum_{l=0}^{p^{r-1}m-1}{\binom{p^{r-1}m+l}{l}}^2{\binom{p^{r-1}m-1}{l}}^2\bigg(\sum_{\lfloor\frac{k}{p}\rfloor=l}\bigg(1-\sum_{j=1,p\nmid j}^{k}\frac{2m^2p^{2r}}{j^2}\bigg)-p\bigg)\\\label{fangc42}&=-2m^2p^{2r}\sum_{l=0}^{p^{r-1}m-1}{\binom{p^{r-1}m+l}{l}}^2{\binom{p^{r-1}m-1}{l}}^2\sum_{\lfloor\frac{k}{p}\rfloor=l}\sum_{j=1,p\nmid j}^{k}\frac{1}{j^2}\pmod{p^{4r}}.
\end{align}
With the help of (\ref{fangc44}), (\ref{fangc23}) and (\ref{fangc42}), for any positive integer $s\leq r-1$, by induction we have
\begin{align}\notag&\sum_{k=0}^{p^{r}m-1}{\binom{p^{r}m+k}{k}}^2{\binom{p^{r}m-1}{k}}^2-p\sum_{k=0}^{p^{r-1}m-1}{\binom{p^{r-1}m+k}{k}}^2{\binom{p^{r-1}m-1}{k}}^2\\\notag&\equiv-2m^2p^{2r}\sum_{l=0}^{p^{r-s}m-1}{\binom{p^{r-s}m+l}{l}}^2{\binom{p^{r-s}m-1}{l}}^2\sum_{\lfloor\frac{k}{p^s}\rfloor=l}\sum_{j=1,p\nmid j}^{k}\frac{1}{j^2}\\\notag&\equiv-2m^2p^{2r}\sum_{l=0}^{p^{r-s-1}m-1}{\binom{p^{r-s-1}m+l}{l}}^2{\binom{p^{r-s-1}m-1}{l}}^2\sum_{\lfloor\frac{k}{p^{s+1}}\rfloor=l}\sum_{j=1,p\nmid j}^{k}\frac{1}{j^2}\\\notag&\equiv-2m^2p^{2r}\sum_{l=0}^{m-1}{\binom{m+l}{l}}^2{\binom{m-1}{l}}^2\sum_{\lfloor\frac{k}{p^r}\rfloor=l}\sum_{j=1,p\nmid j}^{k}\frac{1}{j^2}\equiv 0\pmod{p^{4r}}.
\end{align}
This proves (\ref{fangc41}).

Combining the above, we have completed the proof of Theorem \ref{th1}.
\end{proof}

\begin{proof}[Proof of Theorem \ref{th2}]
We need the following identity due to V. J. W. Guo \cite[(2.6)]{VJWGJZF1}:
\begin{align}\frac{(-1)^n}{n^2}\label{fangc7}\sum_{k=0}^{n-1}(3k+2)(-1)^kf_k=n\sum_{k=1}^{n}{\binom{n-1}{k-1}}^3\frac{n^2-4k^2}{k^3}+1.
\end{align}
Taking $n=p^{r-1+j}m$ for $j\in\{0, 1\}$ in (\ref{fangc7}), where $r, m\in \Z^{+}$ and $p\nmid m$, and noting that  $$\frac{(-1)^{p^rm}-(-1)^{p^{r-1}m}}{p^{r-3}m}=[r=1]\frac{(-1)^{pm}-(-1)^{m}}{m}p^2=[r=1][p=2]\frac{2p^2}{m},$$ we obtain
\begin{align}\notag&\frac{1}{p^{3r-3}m^3}\bigg(\sum_{k=0}^{p^rm-1}(3k+2)(-1)^kf_k-p^2\sum_{k=0}^{p^{r-1}m-1}(3k+2)(-1)^kf_k\bigg)\\\notag&=(-1)^{p^{r}m}p^{3}\sum_{k=1}^{p^rm}{\binom{p^rm-1}{k-1}}^3\frac{p^{2r}m^2-4k^2}{k^3}+[r=1][p=2]\frac{2p^2}{m}\\\label{fangc11}&\ \ \ \ -(-1)^{p^{r-1}m}p^{2}\sum_{k=1}^{p^{r-1}m}{\binom{p^{r-1}m-1}{k-1}}^3\frac{p^{2r-2}m^2-4k^2}{k^3}.
\end{align}
If $r=1$, then we have
\begin{align}\notag&(-1)^{mp}p^{3}\sum_{k=1}^{mp}{\binom{mp-1}{k-1}}^3\frac{m^2p^{2}-4k^2}{k^3}\\\notag&\equiv(-1)^{mp}p^{2}\sum_{k=1}^{m}{\binom{mp-1}{pk-1}}^3\frac{m^2-4k^2}{k^3}
\\\label{fangc9}&=\frac{(-1)^{mp}p^{2}}{m}\sum_{k=1}^{m}\binom{mp}{pk}\bigg({\binom{mp}{pk}}^2-4{\binom{mp-1}{pk-1}}^2\bigg)\pmod{p^{3
}}.
\end{align}
In view of Lemma \ref{l1} and Lemma \ref{l2},
\begin{align}\notag&
\frac{(-1)^{mp}p^{2}}{m}\sum_{k=1}^{m}\binom{mp}{pk}\bigg({\binom{mp}{pk}}^2-4{\binom{mp-1}{pk-1}}^2\bigg)\\\notag&\equiv\frac{(-1)^{mp}p^{2}}{m}\sum_{k=1}^{m}(-1)^{pk-k}\binom{m}{k}\bigg({\binom{m}{k}}^2-4{\binom{m-1}{k-1}}^2\bigg)\\\label{fangc10}&=(-1)^{mp}p^{2}\sum_{k=1}^{m}(-1)^{pk-k}{\binom{m-1}{k-1}}^3\frac{m^2-4k^2}{k^3}\pmod{p^{3
}}.\end{align}
 Combining (\ref{fangc11})-(\ref{fangc10}) with (\ref{fangc3}), it suffices to show that
\begin{align}\notag&(-1)^{mp}p^{2}\sum_{k=1}^{m}(-1)^{pk-k}{\binom{m-1}{k-1}}^3\frac{m^2-4k^2}{k^3}+[p=2]\frac{2p^2}{m}\\\label{fangc8}&\equiv(-1)^{m}p^{2}\sum_{k=1}^{m}{\binom{m-1}{k-1}}^3\frac{m^2-4k^2}{k^3}\pmod{p^{3
}}.\end{align}
It is easy to check (\ref{fangc8}) in the case $p\geq 3$. Next suppose that $p=2$. Since $2\nmid m$,
namely, $m$ is an odd integer, clearly, $\frac{2p^2}{m}\equiv 0\pmod{8}$. By (\ref{fangc8}), we obtain
\begin{align}\notag&4\sum_{k=1}^{m}{\binom{m-1}{k-1}}^3\frac{m^2-4k^2}{k^3}((-1)^k+1)\\\notag&=8\sum_{k=1,2\mid k}^{m}{\binom{m-1}{k-1}}^3\frac{m^2-4k^2}{k^3}=\frac{8}{m}\sum_{k=1,2\mid k}^{m}\binom{m}{k}\bigg({\binom{m}{k}}^2-4{\binom{m-1}{k-1}}^2\bigg)\\\notag&\equiv 0\pmod{8}.
\end{align}
So (\ref{fangc8}) with $p=2$ is concluded. This proves (\ref{fangc3}) in the case $r=1$.

Below we assume $r\geq 2$. By (\ref{fangc3}) and (\ref{fangc11}), it suffices to prove that
\begin{align}\notag
&p^{}\sum_{k=1}^{p^rm}{\binom{p^rm-1}{k-1}}^3\frac{p^{2r}m^2-4k^2}{k^3}\\\label{fangc12}&\ \ \ \ -\sum_{k=1}^{p^{r-1}m}{\binom{p^{r-1}m-1}{k-1}}^3\frac{p^{2r-2}m^2-4k^2}{k^3}\equiv0\pmod{p^{
}}.
\end{align}
Note that
\begin{align}\notag&p^{}\sum_{k=1}^{p^rm}{\binom{p^rm-1}{k-1}}^3\frac{p^{2r}m^2-4k^2}{k^3}-\sum_{k=1}^{p^{r-1}m}{\binom{p^{r-1}m-1}{k-1}}^3\frac{p^{2r-2}m^2-4k^2}{k^3}\\\label{fangc14}&\equiv\sum_{k=1}^{p^{r-1}m}\bigg({\binom{p^rm-1}{pk-1}}^3-{\binom{p^{r-1}m-1}{k-1}}^3\bigg)\frac{p^{2r-2}m^2-4k^2}{k^3}\pmod{p^{
}}.
\end{align}
For $1\leq k\leq p^{r-1}m$, by Lemma \ref{l1} and (\ref{tfangc10}), we have
\begin{align}\notag{\binom{p^rm-1}{pk-1}}^2&\equiv{\binom{p^{r-1}m-1}{k-1}}^2\bigg(1-p^rm\sum_{j=1,p\nmid k}^{pk-1}\frac{1}{j}\bigg)^2\equiv{\binom{p^{r-1}m-1}{k-1}}^2\pmod{p^{r+2}}.
\end{align}
With the help of Lemma \ref{l2} and the above congruence,
\begin{align}\notag&{\binom{p^rm-1}{pk-1}}^3\frac{p^{2r-2}m^2-4k^2}{k^3}\\\notag&=\frac{1}{p^{r-1}m}{\binom{p^rm}{pk}}^3-\frac{4}{p^{r-1}m}\binom{p^rm}{pk}{\binom{p^rm-1}{pk-1}}^2\\\notag&\equiv\frac{(-1)^{pk-k}}{p^{r-1}m}{\binom{p^{r-1}m}{k}}^3-\frac{4(-1)^{pk-k}}{p^{r-1}m}\binom{p^{r-1}m}{k}{\binom{p^{r-1}m-1}{k-1}}^2\\\label{fangc13}&=(-1)^{pk-k}{\binom{p^{r-1}m-1}{k-1}}^3\frac{p^{2r-2}m^2-4k^2}{k^3}\pmod{ p^{r+2-\delta_{p,3}-2\delta_{p,2}}}.\end{align}
The congruence (\ref{fangc12}) in the case $p\geq 3$ easily follows from (\ref{fangc14}) and (\ref{fangc13}).
While $p=2$ and $r\geq 2$, substituing (\ref{fangc13}) into (\ref{fangc14}) yields
\begin{align}\notag&2^{}\sum_{k=1}^{2^rm}{\binom{2^rm-1}{k-1}}^3\frac{2^{2r}m^2-4k^2}{k^3}-\sum_{k=1}^{2^{r-1}m}{\binom{2^{r-1}m-1}{k-1}}^3\frac{2^{2r-2}m^2-4k^2}{k^3}\\\notag&\equiv\sum_{k=1, 2\nmid k}^{2^{r-1}m}((-1)^k-1){\binom{2^{r-1}m-1}{k-1}}^3\frac{2^{2r-2}m^2-4k^2}{k^3}\equiv 0\pmod{2}.
\end{align}
Combining the above, the congruence (\ref{fangc3}) is concluded.

The proof of Theorem \ref{th2} is now complete.
\end{proof}

\begin{Ack}
The author would like to thank the referee for helpful comments.
\end{Ack}

\end{document}